\documentclass[draft,11pt]{conm-p-l-b}
\usepackage{url}
\usepackage{amsmath,amsfonts,amssymb}
\usepackage{amsthm}
\usepackage{multirow}

\title[
    Easy scalar decompositions 
]{
    Easy scalar decompositions
    \\
    for efficient scalar multiplication
    \\
    on elliptic curves and genus 2 Jacobians
}
\author{Benjamin Smith}
\address{
    INRIA \& \'Ecole polytechnique. 
    \'Equipe-Projet GRACE,
    INRIA Saclay -- \^Ile-de-France.
    Laboratoire d'Informatique (LIX),
    1 rue Honor\'e d'Estienne d'Orves,
    Campus de l'\'Ecole polytechnique,
    91120 Palaiseau,
    France
}
\date{\today}

\newcommand{\ZZ}{\mathbb{Z}}
\newcommand{\QQ}{\mathbb{Q}}
\newcommand{\QQbar}{\overline{\QQ}}

\newcommand{\FF}{\mathbb{F}}

\newcommand{\OK}[1][{K}]{\mathcal{O}_{#1}}

\newcommand{\EC}{\mathcal{E}}
\newcommand{\ECK}{\widetilde{\mathcal{E}}}
\newcommand{\Jac}[1]{\ensuremath{\mathcal{J}_{#1}}}
\newcommand{\C}{\mathcal{C}}
\newcommand{\A}{\mathcal{A}}

\newcommand{\G}{\mathcal{G}}

\newcommand{\vv}[1]{\mathbf{#1}}
\newcommand{\Lattice}{\mathcal{L}}

\newcommand{\End}{\mathrm{End}}

\newcommand{\Oh}{O}

\newcommand{\subgrp}[1]{\left\langle{#1}\right\rangle}
\newcommand{\roundoff}[1]{\left\lfloor{#1}\right\rceil}

\newcommand{\dualof}[1]{{{#1}^\dagger}}
\newcommand{\pidual}{\dualof{\pi}}

\newcommand{\conj}[2][\sigma]{{{}^{#1}{#2}}}

\newcommand{\Legendre}[2]{\genfrac{(}{)}{}{}{#1}{#2}}

\newcommand{\Disc}[1]{{\Delta({#1})}}
\newcommand{\Tr}[1]{\ensuremath{t_{#1}}}
\newcommand{\N}[1]{\ensuremath{n_{#1}}}
\newcommand{\tE}[1][\pi]{\Tr{#1}}
\newcommand{\cpol}[1]{P_{#1}}
\newcommand{\cond}[1]{c_{#1}}

\theoremstyle{plain}
\newtheorem{theorem}{Theorem}
\newtheorem{lemma}[theorem]{Lemma}

\theoremstyle{remark}
\newtheorem{remark}[theorem]{Remark}
\theoremstyle{definition}
\newtheorem{definition}{Definition}
\newtheorem{example}{Example}

\begin{document}

\begin{abstract}
    The first step in elliptic curve scalar multiplication algorithms
    based on scalar decompositions using efficient
    endomorphisms---including Gallant--Lambert--Vanstone (GLV) and 
    Galbraith--Lin--Scott (GLS) multiplication, as well as
    higher-dimensional and higher-genus 
    constructions---is to produce a short basis of a certain 
    integer lattice involving the eigenvalues of the endomorphisms.
    The shorter the basis vectors, the shorter the decomposed scalar
    coefficients, and the faster the resulting scalar multiplication.
    Typically, knowledge of the eigenvalues allows us to write down a
    long basis, which we then reduce using the Euclidean algorithm,
    Gauss reduction,
    LLL, or even a more specialized algorithm.

    In this work, 
    we use elementary facts about quadratic rings
    to immediately write down a short basis of the
    lattice for 
    the GLV, GLS, GLV+GLS, and \(\QQ\)-curve constructions
    on elliptic curves,
    and for genus 2 real multiplication constructions.
    We do not pretend that this represents a significant optimization
    in scalar multiplication, since the lattice reduction step
    is always an offline precomputation---but it does give
    a better insight into the structure of scalar decompositions.
    In any case, it is always more convenient 
    to use a ready-made short basis
    than it is to compute a new one.
\end{abstract}

\maketitle

\section{
    Introduction
}
\label{sec:introduction}

Scalar multiplication on elliptic curves (or Jacobians of genus 2
curves) is a key operation in many modern asymmetric cryptographic
primitives.
The classic scenario is as follows:
let \(\G \subset \A(\FF_{q})\)
be a cyclic subgroup of order \(N\),
where
\(\A\) is an elliptic curve or an abelian surface
over a finite field~\(\FF_{q}\).
Given an integer \(m\) (typically on the order of \(N\))
and a point \(P\) in \(\G\),
our goal is to compute 
\[
    [m]P := \underbrace{P + P + \cdots + P}_{m \text{ times }}
\]
as quickly as possible.

Since elliptic curve scalar multiplication is analogous to 
exponentiation in finite fields, many algorithms originally developed 
with the multiplicative groups of finite fields (or general abelian
groups) in mind 
transfer directly to scalar multiplication:
square-and-multiply loops in finite fields
become double-and-add loops on elliptic curves, for example.
However, the geometry of elliptic curves can offer us new 
algorithms with no true finite field analogues.  
A spectacular (and easy) example of this phenomemon is
scalar multiplication with endomorphism decompositions,
originally proposed by 
Gallant, Lambert, and Vanstone~\cite{Gallant--Lambert--Vanstone}.
We present a general version of their
idea below that is flexible enough to accommodate higher-dimensional and
higher-genus constructions.

\subsection*{The general scalar decomposition technique}
\label{sec:general}

Let \(\A\), \(\G\), and \(N\) be as above,
and
let \(\phi_1,\ldots,\phi_r\) be \(\FF_{q}\)-endomorphisms of \(\A\).
We lose nothing by supposing \(\phi_1 = 1\).
In contrast to~\cite{Longa--Sica}, we do \emph{not} suppose that the
\(\phi_i\) form a linearly independent set.

Suppose that \(\phi_i(\G) \subseteq \G\) for \(1 \le i \le r\)
(this is the typical situation in cryptographic applications, where \(N\) is so
close to \(\#\A(\FF_{q})\) that there is no room for the image of \(\G\)
to be anything but \(\G\) itself); then 
each \(\phi_i\) restricts to an endomorphism of \(\G\).  
But \(\G\) is a cyclic group, isomorphic to \(\ZZ/N\ZZ\), 
and as such each of its endomorphisms is multiplication by some integer 
(defined modulo \(N\)).  
In particular, each endomorphism \(\phi_i\) 
acts on \(\G\) as multiplication by an
integer \emph{eigenvalue} \(-N/2 < \lambda_{\phi_i} \le N/2\),
such that 
\[
    \phi_i|_\G = [\lambda_{\phi_i}]_\G
    \ .
\]

\begin{definition}
    Let \(D\) be the \(\ZZ\)-module homomorphism
    \[
        \begin{array}{rrcl}
            D : &
            \ZZ\phi_1 \oplus \cdots \oplus \ZZ\phi_r &
            \longrightarrow &
            \ZZ/N\ZZ
            \\
            &
            \left(a_1,\ldots,a_r\right) &
            \longmapsto &
            a_1\lambda_{\phi_1} + \cdots + a_r\lambda_{\phi_r}
            \ ;
        \end{array}
    \]
    an \(r\)-dimensional\footnote{
        We emphasize that the dimension of \(r\) of a decomposition 
        has no relation to the dimension of \(\A\),
        or to the \(\ZZ\)-rank of the endomorphism ring.
        Typical values for \(r\) are 
        \(r = 1\), corresponding to classical scalar multiplication;
        \(r = 2\), 
        as in in Gallant--Lambert--Vanstone~\cite{Gallant--Lambert--Vanstone}
        and Galbraith--Lin--Scott~\cite{Galbraith--Lin--Scott}
        multiplication;
        and
        \(r = 4\), as proposed by Longa and Sica~\cite{Longa--Sica}
        and Guillevic and Ionica~\cite{Guillevic--Ionica}.
        A technique with \(r = 3\) was proposed by
        Zhou, Hu, Xu, and Song~\cite{Zhou--Hu--Xu--Song},
        but this is essentially Longa--Sica with \(a_4 = 0\).
        Bos, Costello, Hisil, and Lauter
        have implemented a genus 2 scalar multiplication with \(r = 8\), 
        but this seems to be the upper limit of practicality for these
        techniques~\cite{Bos--Costello--Hisil--Lauter}.
    } \emph{decomposition} of a scalar \(m\)  
    is any element of \(D^{-1}(m)\).  
\end{definition}

Returning to the scalar multiplication problem:
we can compute \([m]P\) for any \(P\) in \(\G\) by
using a
multiexponentiation algorithm
on the points \(\phi_1(P),\ldots,\phi_r(P)\)
to compute
\[
    [m]P = [a_1]\phi_1(P) \oplus \cdots \oplus [a_r]\phi_r(P)
    \qquad
    \text{for any}
    \quad
    (a_1,\ldots,a_r) \in D^{-1}(m)
    \ .
\]
    
The literature on exponentation and multiexponentiation algorithms
is vast, and we will not attempt to summarize it here
(but for an introduction to general exponentation and multiexponentation
algorithms, we recommend~\cite[\S2.8,\S11.2]{Galbraith} 
and~\cite[Chapter 9]{Handbook}).
For the purposes of this article,
it suffices to note that
for the scalar decomposition technique to offer 
an advantage over simply computing \([m]P\)
as a conventional exponentiation,
\begin{description}
    \item[The endomorphisms 
        must be \emph{efficient}]
        that is, any \(\phi_i(P)\) must be computable 
        for the cost of a few group operations, and 
    \item[The decomposition 
        must be \emph{short}]
        that is, 
        \[
            \|(a_1,\ldots,a_r)\|_\infty = \max_i|a_i|
        \]
        should be significantly smaller than \(|m|\),
        which is typically on the order of \(N\).
\end{description}

In this article, 
we suppose that we are given a fixed set of efficient \(\phi_i\),
and concentrate on the problem of computing short scalar decompositions.
First, consider the lattice of decompositions of \(0\):
\[
    \Lattice := \ker D = \subgrp{
        (z_1,\ldots,z_r) \in \ZZ^r
        \mid
        z_1\lambda_{\phi_1} + \cdots + z_r\lambda_{\phi_r}
        \equiv 0 \pmod{N}
    }
    \ .
\]
The set of decompositions of any \(m\) in \(\ZZ/N\ZZ\) is then
the lattice coset
\[
    D^{-1}(m) = (m,0,\ldots,0) + \Lattice
    \ .
\]

To find a short decomposition of \(m\),
we can subtract a nearby vector in \(\Lattice\) from \((m,0,\ldots,0)\).
The reference technique for finding such a vector in \(\Lattice\)
is Babai rounding~\cite{Babai}, which works as follows:
if \(\vv{b}_1,\ldots,\vv{b}_r\) is a basis for \(\Lattice\),
then we let \((\alpha_1,\ldots,\alpha_{r})\) be the (unique) solution in
\(\QQ^r\) to the linear system 
\[
    (m,0,\ldots,0) = \sum_{i=1}^r \alpha_i\vv{b}_i 
    \ ,
\]
and set
\[
    (a_1,\ldots,a_r) 
    := 
    (m,0,\ldots,0) - \sum_{i=1}^r \roundoff{\alpha_i}\vv{b}_i
    \ ;
\]
then \((a_1,\ldots,a_r)\) is an \(r\)-dimensional decomposition of
\(m\).
Since 
\[
    (a_1,\ldots,a_r) 
    = 
    \sum_{i=1}^r\left(\alpha_i - \roundoff{\alpha_i}\right)\vv{b}_i
\]
and \(|x - \roundoff{x}| \le 1/2\) for any \(x\) in \(\QQ\),
we have
\[
    \|(a_1,\ldots,a_r)\|_\infty 
    \le 
    \frac{r}{2}\max_i\|\vv{b}_i\|_\infty
    \ .
\]

It is clear, therefore,
that finding short decompositions depends on finding a short basis
for \(\Lattice\).
Note that \(\Lattice\) depends only on the \(\phi_i\), 
and not on the eventual scalars \(m\) or points \(P\) to be multiplied;
as such,
the short basis can (and should) be precomputed.
Assuming that the eigenvalues 
have pairwise differences of absolute value at least \(N^{1/r}\),
there exists a basis with \(\max_i\|\vv{b}\|_\infty\) in \(O(N^{1/r})\),
which will yield scalar decompositions of bitlength around \(\frac{1}{r}\log_2N\).

In most of the scalar decomposition literature,
a short basis of \(\Lattice\) is produced by starting with a long basis
---typically the basis
\begin{align*}
    \vv{b}_1 & = (N,0,\ldots,0) 
    \ , 
    \\
    \vv{b}_2 & = (-\lambda_{\phi_2},1,0,\ldots,0) 
    \ , 
    \\
    \vv{b}_3 & = (-\lambda_{\phi_3},0,1,0,\ldots,0) 
    \ , 
    \\
    & \vdots 
    \\
    \vv{b}_r & = (-\lambda_{\phi_r},0,\ldots,1) 
\end{align*}
---before applying a lattice reduction
algorithm to produce a short basis.
For \(r = 2\), Gallant, Lambert, and Vanstone 
used the Euclidean algorithm,
which is equivalent to the usual Gauss lattice reduction algorithm 
(though Kaib's algorithm for the infinity norm~\cite{Kaib} 
may give marginally better results).
In higher dimensions, we would typically use a fast LLL variant 
(such as \texttt{fpLLL}~\cite{fpLLL})---though 
Longa and Sica~\cite{Longa--Sica} went so 
far as to propose a new 4-dimensional lattice basis reduction algorithm 
for their GLV+GLS construction on elliptic curves.

\subsection*{Our contribution: ready-made short bases}

Our contention in this article is that in most cryptographic situations,
no lattice basis reduction is required to find a short basis
of \(\Lattice\): one can simply write
down vectors of length at most \(O(\#\A(\FF_{q})^{1/r})\) 
from scratch.  The information that allows us to do so is typically a
by-product of the group order computation
(or of the CM method).
For the abelian varieties most useful in cryptography,
these vectors either form a basis for \(\Lattice\),
or can be easily modified to do so.

Galbraith, Lin, and Scott~\cite{Galbraith--Lin--Scott} and the
author~\cite{Smith-QQ} have already constructed families of endomorphisms 
equipped with a convenient ready-made basis; in this work,
we generalize these ready-made bases to all of the other known
efficient endomorphism constructions for elliptic curves and 
to real multiplication techniques for genus 2 Jacobians.
In this way, we construct explicit short bases for
the
Galbraith--Lin--Scott (GLS), 
Gallant--Lambert--Vanstone (GLV), 
Guillevic--Ionica~\cite{Guillevic--Ionica},
Longa--Sica,
and \(\QQ\)-curve reduction techniques,
as well as for the 
Kohel--Smith~\cite{Kohel--Smith}
and Takashima~\cite{Takashima}
methods for genus 2 Jacobians.

We do not pretend that this is a significant optimization for scalar
decomposition methods: the construction of a short basis is essentially
a one-shot precomputation, and existing lattice basis reduction
methods are certainly fast enough on the relevant input sizes.  
However, the construction of these ``instant'' bases
turns out to be an illuminating exercise:
short bases can be read off from
they are simple endomorphism ring relations that are,
in practice, known in advance.
The bottom line is that it is always more
convenient to not compute something than it is to compute it.

\section{
    Relations between quadratic orders
}

We recall some elementary facts from the theory of quadratic fields.
Further details and proofs can be found in almost any basic algebraic
number theory text (we recommend~\cite{Swinnerton-Dyer}).

Let \(K\) be a quadratic field, real or imaginary, 
with maximal order \(\OK\)
and discriminant \(\Delta_K\).
If \(\xi\) is an element of \(\OK\)
then we write 
\(\Tr{\xi}\) for its trace,
\(\N{\xi}\) for its norm.
If \(\xi\) is not in \(\ZZ\),
then it generates an order \(\ZZ[\xi]\) in \(\OK\);
we write \(\Disc{\xi} = \Tr{\xi}^2 - 4\N{\xi}\) 
for the discriminant of \(\ZZ[\xi]\),
and \(\cpol{\xi}(T) = T^2 - \Tr{\xi}T + \N{\xi}\) 
for the minimal polynomial of \(\xi\).
The discriminants of \(\OK\) and \(\ZZ[\xi]\) are related by
\(\Disc{\xi} = \cond{\xi}^2\Delta_K\)
for some positive integer \(\cond{\xi}\),
the conductor of \(\ZZ[\xi]\) in \(\OK\).

The set of orders in \(K\) form a lattice (in the
combinatorial sense), indexed by the conductor: \(\ZZ[\xi] \subset
\ZZ[\xi']\) if and only if \(\cond{\xi'} \mid \cond{\xi}\).
If \(\ZZ[\xi] \subset \ZZ[\xi']\) are orders in \(K\),
then necessarily
\begin{equation}
    \label{eq:relation-1}
    \xi = c\xi' + b 
\end{equation}
for some integers \(b\) and \(c\).
It follows that
\begin{equation}
    \label{eq:b-and-c}
    b = \frac{1}{2}\left(\Tr{\xi} - c\Tr{\xi'}\right)
    \qquad
    \text{and}
    \qquad
    c^2 = \frac{\Disc{\xi}}{\Disc{\xi'}}
    \ .
\end{equation}
Note that \(c\) is, up to sign, the relative conductor of
\(\ZZ[\xi]\) in \(\ZZ[\xi']\).
Multiplying Eq.~\eqref{eq:relation-1}
through by \(\Tr{\xi'} - \xi'\),
which is also \(\N{\xi'}/\xi'\),
we obtain a second relation
\begin{equation}
    \label{eq:relation-2}
    \xi\xi' - \Tr{\xi'}\xi - b\xi' + (c\N{\xi'} + b\Tr{\xi'}) 
    = 
    0
    \ .
\end{equation}
The following lemma turns the relations between endomorphisms
of Eqs.~\eqref{eq:relation-1} and~\eqref{eq:relation-2}
into relations between eigenvalues, which we will use later to
produce short lattice vectors.

\begin{lemma}
    \label{lemma:relations}
    Let \(\xi\) and \(\xi'\) be 
    endomorphisms of an abelian variety \(\A/\FF_{q}\)
    such that \(\ZZ[\xi]\) and \(\ZZ[\xi']\) are quadratic rings
    and \(\ZZ[\xi] \subseteq \ZZ[\xi']\),
    so
    \(\xi = c\xi' + b\)
    for some integers \(b\) and \(c\).
    Let \(\G\subset\A\) be a cyclic subgroup of order \(N\)
    such that \(\xi(\G) \subseteq \G\) and \(\xi'(\G) \subseteq \G\),
    and let \(\lambda\) and \(\lambda'\) be the eigenvalues in
    \(\ZZ/N\ZZ\) of \(\xi\) and \(\xi'\) on \(\G\), respectively.
    Then 
    \begin{align*}
        \lambda - c\lambda' - b 
        & \equiv 0 \pmod{N}
        \quad \text{and}
        \\
        \lambda\lambda' - \Tr{\xi'}\lambda - b\lambda' +
        c\N{\xi'} + b\Tr{\xi'}
        & \equiv 0 \pmod{N}
        \ .
    \end{align*}
\end{lemma}
\begin{proof}
    This follows immediately by mapping the relations in
    Eqs.~\eqref{eq:relation-1} and~\eqref{eq:relation-2}
    through the homomorphism \(\ZZ[\xi']\to\End(\G) \cong \ZZ/N\ZZ\)
    sending \(\xi'\) to \(\lambda' \pmod{N}\)
    (and \(\xi\) to \(\lambda \pmod{N}\)).
\end{proof}

\section{
    General two-dimensional decompositions for elliptic curves
}

Let \(\EC/\FF_{q}\) be an ordinary elliptic curve.
If \(\pi\) is the \(q\)-power Frobenius endomorphism on \(\EC\)
then
\[
    \cpol{\pi}(T) = T^2 - \tE T + q
    \ ,
\]
where
\[
    |\tE| \le 2\sqrt{q}
    \qquad
    \text{and}
    \qquad
    \Delta_\pi :=
    \tE^2 - 4q < 0
    \ .
\]

\begin{theorem}
    \label{th:EC-2d}
    Let \(\phi\) be a non-integer endomorphism of \(\EC\)
    such that \(\ZZ[\pi] \subset \ZZ[\phi]\),
    so \(\pi = c\phi + b\) for some integers \(c\) and \(b\).
    Suppose that we are in the situation of \S\ref{sec:general}
    with \(\A = \EC\) and \((\phi_1,\phi_2) = (1,\phi)\).
    The vectors
    \[
        \vv{b}_1 
        =
        \left(
            b - 1 , c
        \right)
        \qquad
        \text{and}
        \qquad
        \vv{b}_2 
        =
        \left(
            c\deg(\phi) + (b-1)\Tr{\phi}, 1 - b
        \right)
    \]
    generate a sublattice 
    of \(\Lattice\)
    of determinant \(\#\EC(\FF_{q})\).
    If \(\G = \EC(\FF_{q})\),
    then \(\Lattice = \subgrp{\vv{b}_1,\vv{b}_2}\).
\end{theorem}
\begin{proof}
    The Frobenius endomorphism \(\pi\)
    fixes the points in \(\EC(\FF_{q})\),
    so \(\pi(\G) = \G\) and 
    the eigenvalue of \(\pi\) on \(\G\) is \(\lambda_{\pi} = 1\).
    Applying Lemma~\ref{lemma:relations}
    with \((\xi,\xi') = (\pi,\phi)\) and \((\lambda,\lambda') =
    (1,\lambda_{\phi})\),
    we obtain relations
    \begin{align*}
        (b-1)\cdot 1 + c\cdot\lambda 
        & \equiv 0 \pmod{N} \qquad \text{and}
        \\
        \left((b - 1)\Tr{\phi} + c\N{\phi}\right)\cdot 1
        + 
        (1 - b)\cdot\lambda 
        & \equiv 0 \pmod{N}
        \ .
    \end{align*}
    The first implies that \(\vv{b}_1\) is in \(\Lattice\),
    the second that \(\vv{b}_2\) is in \(\Lattice\).
    Equations~\eqref{eq:b-and-c} imply that
    \[
        \det(\subgrp{\vv{b}_1,\vv{b}_2}) 
        = 
        (b-1)^2 + c(c\deg(\phi) + (b-1)\Tr{\phi}) 
        = 
        q - \tE + 1 
        \ ,
    \]
    which is \(\#\EC(\FF{q})\).
\end{proof}

Theorem~\ref{th:EC-2d} constructs two basis vectors, 
but it makes no claim about their length.
We can give some almost-trivial bounds on the size of
\(b\) and \(c\),
using
\(4q - \tE^2 = -c^2\Disc{\phi} \),
\(b = \tfrac{1}{2}\left(\tE - c\Tr{\phi}\right)\),
and the triangle inequality:
\begin{equation}
    \label{eq:trivial-bounds}
    |c| = 2\sqrt{q}\sqrt{((\tE/2\sqrt{q})^2-1)/\Disc{\phi}}
    \qquad
    \text{and}
    \qquad
    |b| \le \frac{1}{2}|\tE| + \frac{1}{2}|\Tr{\phi}||c|
    \ .
\end{equation}

But proving general bounds that apply for arbitrary endomorphisms
is probably the wrong approach
when all of the endomorphisms used in practical scalar decompositions are 
special (and deliberately non-general) constructions.
Consider the Ciet--Sica--Quisquater bounds
of~\cite[Theorem~1]{Ciet--Sica--Quisquater} for the case \(r = 2\)
of \S1:
for every integer \(m\),
there exists a decomposition \((a_1,a_2)\) of \(m\)
such that
\(\|(a_1,a_2)\|_\infty \le C \sqrt{N}\),
with \(C = (1 + |\Tr{\phi}| + \N{\phi})^{1/2}\).
In terms of bitlength,
\begin{equation}
    \label{eq:Ciet--Sica--Quisquater-bound}
    \log_2\|(a,b)\|_\infty
    \le 
    \tfrac{1}{2}\log_2N 
    + 
    \tfrac{1}{2}\log_2\left(1 + |\Tr{\phi}| + \N{\phi}\right) 
    \ .
\end{equation}
In this theorem, \(t_\phi\) and \(n_\phi\) (and \(C\))
are implicitly treated as constants---that is, 
independent of \(N\) and \(q\).
While this is appropriate for GLV curves, 
\(t_\phi\) and \(n_\phi\) are not ``constant'' when 
\(\phi\) is inseparable (notably, \(n_\phi\) is divisible by \(p\)).
In this case, the bound above is spectacularly loose:
Remark~\ref{rem:GLS-bounds} below gives a detailed example of this
in the context of GLS endomorphisms.

\section{Shrinking the basis (or expanding the sublattice) to fit~\(\G\)}
\label{sec:shrink}

Let \(h\) be the cofactor such that \(\#\EC(\FF_{q}) = hN\).
By the Pohlig--Hellman--Silver reduction~\cite{Pohlig--Hellman},
the cryptographic strength of \(\EC\) 
depends entirely on the size of the prime \(N\),
so we should choose \(\EC\) with \(N\) as large as possible.
While \(h = 1\) is ideal, 
allowing \(h = 2\) or \(4\)
permits faster curve arithmetic 
via transformations to Montgomery~\cite{Montgomery} 
or twisted Edwards~\cite{Hisil--Wong--Carter--Dawson}
models, for example.
(In the pairing-based context \(h\) may be somewhat larger.)

The lattice \(\Lattice\)
has determinant \(N\),
but the vectors \(\vv{b}_1,\vv{b}_2\)
constructed by Theorem~\ref{th:EC-2d}
are a basis for a (sub)lattice of determinant \(hN\).
If \(h \not= 1\),
then while our sublattice will still give short decompositions,
it is suboptimal.
If \(h = 4\), for example,
then our basis may be one bit too long---which 
generally means one double too many when the resulting decompositions
are used in a multiexponentiation algorithm.

If for some reason \(\G \not= \EC(\FF_{q})\),
then we want to be able to derive a basis of the full lattice \(\Lattice\)
from \(\vv{b}_1,\vv{b}_2\).
First, note that 
\[
    \subgrp{
        \tfrac{1}{g}\vv{b}_1 \ , 
        \tfrac{1}{g}\vv{b}_2
    }
    \subseteq
    \Lattice
    \qquad
    \text{where}
    \qquad
    g = \gcd(c,b-1)
    \ ,
\]
because the relations 
\((b-1) + c\lambda \equiv 0 \pmod{N}\)
and
\(c\deg(\phi) + (b-1)\Tr{\phi} + (b-1)\lambda \equiv 0 \pmod{N}\)
still hold when we divide through by \(g\).
This new sublattice has index \(\#\EC(\FF_{q})/(g^2N)\) in \(\Lattice\).
Note that if \(g \not= 1\),
then \(\EC[g] \subset \EC(\FF_{q})\),
because 
\(\pi - 1 = g((c/g)\phi + (b-1)/g)\),
so \(\pi - 1\) factors through \([g]\).

More generally,
suppose \(\ell\) is a prime dividing \(h\):
then there exists a sublattice 
\(\Lattice'\) such that
\[
    \subgrp{\vv{b}_1,\vv{b}_2}
    \subset 
    \Lattice'
    \subseteq 
    \Lattice
    \qquad
    \text{with}
    \qquad
    [\Lattice':\subgrp{\vv{b}_1,\vv{b}_2}] = \ell
    \ .
\]
Looking at the components of \(\vv{b}_1\) and \(\vv{b}_2\),
we see that if \(\ell\) divides \(c\), then it must also divide \(b-1\)
(otherwise \(\Lattice'\) cannot exist),\footnote{
    It is possible to take a much more highbrow view of all this:
    Theorem~1 of~\cite{Lenstra} implies that if \(\EC[\ell] \subseteq
    \EC(\FF_{q})\) then \(\ell\) must divide the conductor of \(\ZZ[\pi]\)
    in the endomorphism ring of \(\EC\) 
    (see also~\cite[p.~40]{Kohel}).
    In this case \(\ell\) divides both \(\pi - 1\) and \(c\phi\),
    so it must also divide \(b-1\).
    Removing the factor of \(\ell^2\) from the index of our sublattice
    therefore corresponds to removing the contribution of the
    full \(\ell\)-torsion to our endomorphism relations.
}
and we can replace each \(\vv{b}_i\) with \(\frac{1}{\ell}\vv{b}_i\) as
above to produce a sublattice of index \(h/\ell^2\) in \(\Lattice\).

Suppose now that \(\ell\) does not divide \(c\).
If \(\ell\) divides \(b-1\) and \(\deg\phi\),
then \(\frac{1}{\ell}\vv{b}_2\) is in \(\Lattice\),
so we can take \(\Lattice' = \subgrp{\vv{b}_1,\frac{1}{\ell}\vv{b}_2}\).
Otherwise, \(\frac{1}{\ell}(\vv{b}_1 + i\vv{b}_j)\) is in
\(\Lattice\) for precisely one \(0 < i < \ell\):
that is, \(i = -c(b-1)^{-1} \bmod{\ell}\).
We can therefore take
\(\Lattice' = \subgrp{\vv{b}_1,\frac{1}{\ell}(\vv{b}_1 + i\vv{b}_2)}\).

Iterating this process factor by factor of \(h\), 
we can gradually shrink the vectors produced by
Theorem~\ref{th:EC-2d} to derive a true basis of \(\Lattice\).  But
as we remarked above, in conventional discrete-log based cryptography the
most important cases are \(h = 1, 2\), and \(4\), and then there is almost
nothing to be done.  
To handle the cases where \(\EC(\FF_{q}) \cong \G\oplus\ZZ/2\ZZ\) or
\(\G\oplus(\ZZ/2\ZZ)^2\), for example, the following
simple procedure produces a basis for~\(\Lattice\):
\begin{itemize}
    \item   If \(c\) is even, then 
        \(\Lattice = \subgrp{\frac{1}{2}\vv{b}_1,\frac{1}{2}\vv{b}_2}\).
    \item   If \(c\) and \(b\) are odd and \(\deg\phi\) is even, then
        \(\Lattice = \subgrp{\vv{b}_1,\frac{1}{2}\vv{b}_2}\).
    \item   Otherwise,
        \(\Lattice = \subgrp{\vv{b}_1,\frac{1}{2}(\vv{b}_1 +
        \vv{b}_2)}\).
\end{itemize}

\section{Decompositions for GLV endomorphisms}

Let \(\widetilde{\EC}\) be an elliptic curve over \(\QQbar\)
with complex multiplication by \(\ZZ[\sqrt{\Delta}]\);
that is, with an explicit endomorphism \(\widetilde{\phi}\)
such that \(\ZZ[\widetilde{\phi}] \cong \ZZ[\sqrt{\Delta}]\).
Let \(\EC/\FF_{q}\) be the (good) reduction modulo \(p\) of \(\EC\),
(and suppose that \(\EC\) is not supersingular);
by definition, \(\EC\) comes equipped with 
an explicit separable endomorphism \(\phi\)
such that \(\ZZ[\phi] \cong \ZZ[\sqrt{\Delta}]\).
If \(\Delta\) and \(\widetilde{\EC}\) were
was chosen in such a way that \(\phi\) has very low degree,
then \(\phi\) can be efficiently computable,
and hence useful for scalar decompositions.

Suppose, therefore,
that \(\phi\) has very low degree for efficiency reasons.
Then \(\N{\phi} = \deg(\phi)\) 
must be very small;
and since \(\Disc{\phi}\) must be negative,
\(\Tr{\phi}^2\) (and hence \(|\Disc{\phi}|\))
must also be very small.
In particular, \(|\Tr{\phi}| < 2\sqrt{\deg(\phi)}\)
and \(|\Disc{\phi}| < 4\deg(\phi)\).

In practice, \(\ZZ[\phi]\) is either the maximal order in
\(\QQ(\pi)\), or (exceptionally) an order of index two in the maximal order.
It is therefore reasonable to 
\begin{center}
    assume \(\ZZ[\pi]\) is contained in \(\ZZ[\phi]\),
\end{center}
so
\[
    \pi = c\phi + b
    \qquad
    \text{with}
    \qquad
    b = \tfrac{1}{2}\left(\tE - c\Tr{\phi}\right)
    \qquad
    \text{and}
    \qquad
    c^2\Disc{\phi} = \Tr{\pi}^2 - 4q
    \ .
\]
This allows us to write down a basis for (a sublattice of) the GLV lattice
using Theorem~\ref{th:EC-2d}.

It is important to note that in practice, 
\(b\) and \(c\)
are already known from the determination of the curve order,
precisely because \(\pi = c\phi + b\).
Indeed,
if we want to compute the order of an elliptic curve known to have
complex multiplication by a CM order (such as \(\ZZ[\phi]\), in this
case),
then we would typically use
the algorithm described in~\cite[\S4]{Schoof},
which
computes the trace \(\tE\) of Frobenius
\emph{by computing \(b\) and \(c\)}.
This approach uses Cornacchia's algorithm
to compute a generator of a principal ideal in \(\ZZ[\phi]\) of norm \(q\),
before taking its trace to compute \(\#\EC(\FF_{q})\);
but this generator is none other than \(\pi = c\phi + b\).

Alternatively,
\(\EC\) could be constructed using the CM method
starting from the small discriminant \(\Disc{\phi}\).
In this case, \(b\) and \(c\) are explicitly constructed so that 
\(c\phi + b\) will have norm \(q\).

In any case, 
if we had somehow mislaid the values of \(b\) and \(c\),
then we could recover them by factoring the ideal \((q)\) in
\(\ZZ[\phi]\) using (for example) Cornacchia's algorithm,
which amounts to repeating the point counting algorithm described above.
The element \(c\phi + b\) will be (up to sign) one of the two resulting 
generators of the factors of \((q)\).
Alternatively, 
we could use \(c \equiv (2-\Tr{\pi})/(2\lambda_\phi - \Tr{\phi}) \pmod{N}\);
though inverting \(2\lambda_\phi - \Tr{\phi}\) modulo \(N\) is roughly
equivalent to the use of the Euclidean algorithm in
the original GLV method.

One important feature of the GLV setting is that \(\Tr{\phi}\),
\(\N{\phi}\), and
\(\Disc{\phi}\) are independent of \(q\) and \(\tE\),
so the bitlength of the basis produced by Theorem~\ref{th:EC-2d}
exceeds \(\frac{1}{2}\log_2q\) by no more than an explicit constant.
The following examples
consider the new basis for the GLV curves with endomorphisms
of degree at most \(3\) 
(treated in detail elsewhere 
by Gallant, Lambert, and Vanstone~\cite{Gallant--Lambert--Vanstone}
and Longa and Sica~\cite{Longa--Sica}).

\begin{example}[\(j\)-invariant 1728:
    \protect{cf.~\cite[Ex.~3]{Gallant--Lambert--Vanstone} and~\cite[Ex.~1]{Longa--Sica}}]
    If \(q \equiv 3 \pmod{4}\),
    then for every \(a \not=0\) in \(\FF_{q}\)
    the curve \(\EC_{1728}: y^2 = x^3 + ax\)
    has an \(\FF_{q}\)-endomorphism \(\phi: (x,y) \mapsto (-x,-iy)\)
    (where \(i^2 = -1\)),
    with \(\cpol{\phi}(T) = T^2 + 1\);
    so \(\ZZ[\phi] \cong \ZZ[\sqrt{-1}]\).
    Theorem~\ref{th:EC-2d}
    constructs the basis
    \[
        \vv{b}_1 = \left(\tfrac{1}{2}\tE-1,c\right)
        \qquad
        \text{and}
        \qquad
        \vv{b}_2 = \left(c,1-\tfrac{1}{2}\tE\right)
        \ ,
    \]
    where
    \(c^2 = q-(\tE/2)^2\).
    This basis is not only short 
    (clearly \(\|\vv{b}_1\|_\infty = \|\vv{b}_2\|_\infty \le
    \tfrac{1}{2}\log_2q\)) and reduced, 
    it is also orthogonal.
\end{example}

\begin{example}[\(j\)-invariant 0:
    \protect{cf.\cite[Ex.~4]{Gallant--Lambert--Vanstone} and~\cite[Ex.~2]{Longa--Sica}}]
    If \(q \equiv 2 \pmod{3}\),
    then for any \(a \not=0\) in \(\FF_{q}\),
    the curve \(\EC_{1728}: y^2 = x^3 + a\)
    has an \(\FF_{q}\)-endomorphism \(\phi: (x,y) \mapsto (\zeta_3x,y)\)
    (where \(\zeta_3\) is a primitive third root of unity),
    with \(\cpol{\phi}(T) = T^2 + T + 1\):
    that is, \(\EC_{0}\)
    has explicit CM by \(\ZZ[(1 + \sqrt{-3})/2]\).
    Looking at the basis produced by Theorem~\ref{th:EC-2d},
    we find
    \[
        \vv{b}_1 = \left( \tfrac{1}{2}(\tE - c) - 1 , c \right)
        \quad
        \text{and}
        \quad
        \vv{b}_2 
        = 
        \left( \tfrac{1}{2}(\tE + c) - 1, 1-\tfrac{1}{2}(\tE - c)  \right)
        \ ,
    \]
    where \(c^2 = \tfrac{1}{3}\left(4q - \tE^2\right)\).
    We note that in this case,
    two applications of the triangle inequality yields
    \(\log_2\|\vv{b}_i\|_\infty < \frac{1}{2}\log_2 q + 1\).
\end{example}

\begin{example}[\(j\)-invariant \(-3375\):
    \protect{cf.~\cite[Ex.~5]{Gallant--Lambert--Vanstone} and~\cite[Ex.~3]{Longa--Sica}}]
    Suppose \(-7\) is a square in \(\FF_{q}\).
    The curve 
    \(\EC_{-3375}: y^2 = x^3 - \tfrac{3}{4}x^2 - 2x  - 1\) 
    over \(\FF_{q}\)
    has a degree-2 endomorphism \(\phi\)
    with \(\cpol{\phi}(T) = T^2 - T + 2\)
    (and \(\ker\phi = \subgrp{(2,0)}\));
    that is, \(\EC_{-3375}\) 
    has explicit CM by \(\ZZ[(-1+\sqrt{-7})/2]\).
    Theorem~\ref{th:EC-2d} yields vectors
    \[
        \vv{b}_1 = \left( b - 1, c \right)
        \qquad
        \text{and}
        \qquad
        \vv{b}_2 = \left(2c - (b-1), (1-b) \right)
        \ ;
    \]
    as before, \(\log_2\|\vv{b}_i\|_\infty < \tfrac{1}{2}\log_2q + 1\).
\end{example}

\begin{example}[\(j\)-invariant \(8000\): \protect{cf.~\cite[Ex.~6]{Gallant--Lambert--Vanstone} and~\cite[Ex.~4]{Longa--Sica}}]
    Suppose \(-2\) is a square in \(\FF_{q}\).
    The curve \(\EC_{8000}: y^2 = 4x^3 - 30x - 28\)
    over \(\FF_{q}\)
    has a degree-2 endomorphism \(\phi\) with 
    \(\cpol{\phi}(T) = T^2 + 2\)
    (and \(\ker\phi = \subgrp{(-2,0)}\)):
    that is, 
    \(\EC_{8000}\) has explicit CM by \(\ZZ[\sqrt{-2}]\).
    Theorem~\ref{th:EC-2d}
    yields vectors
    \[
        \vv{b}_1 = \left( b - 1, c \right)
        \qquad\text{and}\qquad
        \vv{b}_2 = \left( 2c, 1 - b \right)
    \]
    with \(\log_2\|\vv{b}_i\| \le \frac{1}{2}\log_2q + 1\).
\end{example}

\begin{example}[\(j\)-invariant 32768: \protect{cf.~\cite[Ex.~5]{Longa--Sica}}]
    Suppose \(-11\) is a square in \(\FF_{q}\).
    The curve
    \(\EC_{32768}: y^2 = x^3 - \tfrac{13824}{539}x +
    \tfrac{27648}{539}\)
    over \(\FF_{q}\)
    has a degree 3 endomorphism \(\phi\) 
    with \(\cpol{\phi}(T) = T^2 - T + 3\):
    that is, 
    \(\EC_{32768}\) has explicit CM by \(\ZZ[\frac{1}{2}(1 + \sqrt{-11})]\).
    The kernel of \(\phi\) 
    is cut out by \((x - \frac{24}{7}(1 - 1/\sqrt{-11}))\).
    Theorem~\ref{th:EC-2d}
    constructs a pair of vectors
    \[
        \vv{b}_1 = \left(b-1,c\right)
        \qquad\text{and}\qquad
        \vv{b}_2 = \left(3c-(b-1),1-b\right)
    \]
    with \(\log_2\|\vv{b}_i\|_\infty < \frac{1}{2}\log_2q + 2\).
\end{example}

\begin{example}[\(j\)-invariant 54000: \protect{cf.~\cite[Example 6]{Longa--Sica}}]
    Suppose \(-3\) is a square in \(\FF_{q}\).
    The curve 
    \(\EC_{54000}: y^2 = x^3 - \tfrac{3375}{121}x + \tfrac{6750}{121}\)
    over \(\FF_{q}\)
    has an \(\FF_{q}\)-endomorphism 
    \(\phi\) of degree \(3\) 
    with minimal polynomial \(\cpol{\phi}(T) = T^2 + 3\)
    (and kernel cut out by \((x-45/11)\)):
    that is, \(\EC_{54000}\) has explicit CM by \(\ZZ[\sqrt{-3}]\).
    Theorem~\ref{th:EC-2d}
    yields vectors
    \[
        \vv{b}_1 = \left(b-1,c\right)
        \qquad\text{and}\qquad
        \vv{b}_2 = \left(3c,1-b\right)
    \]
    with \(\log_2\|\vv{b}_i\|_\infty < \tfrac{1}{2}\log_2q + 1\).
\end{example}

\section{Decompositions for the GLS endomorphism}

\label{sec:GLS}

Let \(\EC_0\) be an ordinary elliptic curve over \(\FF_p\),
and
let \(\EC := \EC_0\times\FF_{p^2}\) be the base extension of \(\EC_0\)
to \(\FF_{p^2}\).  
The Frobenius endomorphism \(\pi_0\) of \(\EC_0\)
has characteristic polynomial \(\cpol{\pi}(T) = T^2 - t_0T + p\)
with \(|t_0| < 2\sqrt{p}\),
while
the (\(p^2\)-power) Frobenius endomorphism \(\pi\) of \(\EC\) 
satisfies \(\pi = \pi_0^2\), 
so \(\cpol{\pi}(T) = T^2 - (t_0^2 - 2p)T + p^2\).
In particular, 
\[
    \#\EC(\FF_{p^2}) 
    = 
    \cpol{\pi}(1) 
    = 
    (p + 1)^2 - t_0^2 = \#\EC_0(\FF_p)\cdot(p + 1 + t_0)
    ,
\]
so \(\#\EC(\FF_{p^2})\) cannot have prime divisors larger than
\(\Oh(p)\).

Now let \(\EC'\) be the quadratic twist of \(\EC\),
with \(\tau: \EC \to \EC'\) the twisting isomorphism.
The Frobenius \(\pi'\) on \(\EC'\) satisfies
\(\pi' = \tau\pi\tau^{-1}\), so
\[
    \cpol{\pi'}(T) 
    = 
    T^2 - (2p-t_0^2)T + p^2  
    \qquad
    \text{and}
    \qquad
    \Delta_{\pi'} = t_0^2(t_0^2 - 4p)
    \ .
\]
Note that 
\( \#\EC'(\FF_{q}) = \cpol{\pi'}(1) = (p-1)^2 + t_0^2 \);
unlike \(\#\EC(\FF_{q})\),
this can take prime (and near-prime) values,
so \(\EC'\) may be useful for discrete-logarithm-based cryptosystems.

The GLS endomorphism on \(\EC'\) is \(\psi := \tau\pi_0\tau^{-1}\).
It is defined over \(\FF_{p^2}\);
its minimal polynomial is
\[
    \cpol{\psi}(T) = \cpol{\pi_0}(T) = T^2 - t_0T + p 
    \ ,
    \qquad
    \text{and}
    \qquad
    \Delta_\psi = t_0^2 - 4p
    \ .
\]
If \(\G\) is a cyclic subgroup of \(\EC'(\FF_{p^2})\)
of order \(N\) 
such that \(\psi(\G) \subseteq \G\),
then the eigenvalue \(\lambda_\psi\) of \(\psi\) on \(\G\) 
is a square root of \(-1\) modulo \(N\).

We have \(\psi^2 = \tau\pi_0^2\tau^{-1} = \tau\pi\tau^{-1} = -\pi'\),
so \(\ZZ[\psi]\) contains \(\ZZ[\pi']\).
We can therefore apply Theorem~\ref{th:EC-2d}
to the inclusion \(\ZZ[\pi'] \subseteq \ZZ[\psi]\)
in order to compute a short basis for (a sublattice of)
\(\Lattice = \subgrp{(N,0),(-\lambda_\psi,1)}\).

Looking at the discriminants, 
we see that \(\ZZ[\pi']\) has conductor \(|t_0|\) in \(\ZZ[\psi]\).
Indeed, 
\[
    \pi' = -t_0\psi + p 
    \ ;
\]
so Theorem~\ref{th:EC-2d} yields a basis
\[
    \vv{b}_1 = (p-1,-t_0)
    \qquad
    \text{and}
    \qquad
    \vv{b}_2 = (-t_0,1-p)
    \ .
\]
This is precisely (up to sign) the basis
of~\cite[Lemma~3]{Galbraith--Lin--Scott};
it is not only short (the largest coefficient is \(p-1\),
so \(\log_2\|\vv{b}_i\|_\infty < \log_2p\)), 
it is also orthogonal.
If \(\EC'(\FF_{p^2})\) does not have prime order,
then we can easily shrink the basis to fit \(\G\)
by following the procedure described in \S\ref{sec:shrink}.

\begin{remark}
    \label{rem:GLS-bounds}
    From a purely formal point of view,
    we could have treated this identically to the GLV
    case, with \(\psi\) in place of \(\phi\),
    but there are a number of important differences.
    First of all, the ring \(\ZZ[\psi]\) has a much larger discriminant
    than any GLV order:
    in general \(\ZZ[\psi]\) is far from being the maximal order of the
    endomorphism algebra.
    Second, the parameters \(\Tr{\psi}\), \(\N{\psi}\), 
    and \(\Disc{\psi}\)
    vary with \(p\) and \(t\),
    so we cannot treat the excess bitlength in
    the Ciet--Sica--Quisquater bounds
    (Ineq.~\eqref{eq:Ciet--Sica--Quisquater-bound})
    as a constant.  
    Indeed, if we simply plug the values
    \begin{align*}
        q & = p^2
        \ , &
        \tE & = 2p - t_0^2
        \ , &
        \Disc{\pi} & = t_0^2(t_0^2 - 4p)
        \ ,
        \\
        \N{\psi} & = p
        \ , &
        \Tr{\psi} & = t_0
        \ , &
        \Disc{\psi} & = t_0^2 - 4p
    \end{align*}
    into Inequalities~\eqref{eq:trivial-bounds}
    or~\eqref{eq:Ciet--Sica--Quisquater-bound},
    or even the ``optimal'' bound of~\cite[Theorem~4]{Ciet--Sica--Quisquater},
    then we obtain a rather pessimistic bitlength bound of around
    \(\frac{3}{4}\log_2 q\), which exceeds the true length of the basis
    by \(\frac{1}{4}\log_2q\) bits.
\end{remark}

\section{Decompositions for reductions of \(\QQ\)-curves}

GLS curves may be seen as a special case of a more general construction
involving reductions of quadratic \(\QQ\)-curves.
We give a very brief description of this construction here
(see~\cite{Smith-QQ} for more details, and families of examples).
In~\cite[Proposition~2]{Smith-QQ}, 
a basis is constructed in a seemingly ad-hoc way that yields half-length
scalar decompositions; we will see below that this basis also
results from Theorem~\ref{th:EC-2d}.

Let \(K\) be a quadratic field, \(\sigma\) the nontrivial automorphism of \(K\)
fixing \(\QQ\).
Let
\(\ECK : y^2 = x^3 + ax + b\)
and \(\conj[\sigma]{\ECK} : y^2 = x^3 + \sigma(a)x + \sigma(b)\)
be a pair of Galois-conjugate curves over \(K\)
such that there exists an isogeny
\(\widetilde{\phi}: \EC \to {}^{(p)}\EC\)
of small degree \(d\)
defined over \(K(\sqrt{-d})\).
If \(p\) is an inert prime in \(K\)
that is a prime of good reduction for \(\ECK\) (and not dividing \(d\)),
then we can reduce \(\widetilde{\phi}: \ECK \to \conj[\sigma]{\ECK}\)
modulo \(p\)
to obtain a \(d\)-isogeny \(\phi: \EC \to \conj[(p)]{\EC}\)
of curves over \(\FF_{p^2}\).
Here \(\conj[(p)]{\EC}\), the reduction of \(\conj[\sigma]{\EC}\) modulo
\(p\), is the curve formed from \(\EC\) by applying
\(p\)-th powering to its coefficients;
so there also exists a \(p\)-th power Frobenius isogeny
\(\pi_0: \conj[(p)]{\EC} \to \EC\).
Composing \(\phi\) with \(\pi_0\), 
we obtain an inseparable endomorphism
\(\psi := \pi_0\phi\) of \(\EC\), of degree \(dp\).
If \(d\) is very small, then \(\psi\) can be efficiently computable,
since \(p\)-th powering in \(\FF_{p^2}\) is essentially free.
(The GLS construction is equivalent to the special case where \(\phi\)
is an isomorphism---that is, \(d = 1\).)

Let \(\epsilon_p := -\Legendre{-d}{p}\);
that is, \(\epsilon_p = 1\) if \(-d\) is a nonsquare modulo \(p\),
and \(-1\) if it is a square.
Then according to~\cite[Proposition~1]{Smith-QQ},
the minimal polynomial of \(\psi\) is 
\[
    \cpol{\psi}(T) = T^2 - \epsilon_prdT + dp
    \ ,
\]
and 
\[
    dr^2 = 2p + \epsilon_p\tE
    \ .
\]
(This determines \(r\) up to sign; exchanging \(r\) with \(-r\)
corresponds to exchanging \(\phi\) with \(-\phi\).)
Squaring the endomorphism,
we obtain 
\(\psi^2 = \pi_0\phi\pi_0\phi = \epsilon_p\pi_0\phi\dualof{\phi}\pi_0 = \epsilon_p[d]\pi\),
where \(\pi\) is the usual \(p^2\)-power Frobenius on \(\EC\).
In particular,
\(\ZZ[\pi]\) is contained in \(\ZZ[\psi]\);
we find that
\[
    \pi = r\psi - \epsilon_p p 
    \ .
\]
If \(\G \subseteq \EC(\FF_{q})\) is a cyclic subgroup of order \(N\)
such that \(\psi(\G) = \G\), then the eigenvalue of \(\psi\) on \(\G\) is
a square root of \(\epsilon_p d\) modulo \(N\).

Applying Theorem~\ref{th:EC-2d} to the orders
\(\ZZ[\pi]\subseteq\ZZ[\psi]\),
we obtain the basis
\[
    \vv{b}_1 
        =
        \left(
            -(1 + \epsilon_p p), r
        \right)
    \qquad
    \text{and}
    \qquad
    \vv{b}_2 
        =
        \left(
            -\epsilon_p rd, 1 + \epsilon_p p
        \right)
    \ .
\]
Up to sign, this is the basis appearing in the proof
of~\cite[Proposition~2]{Smith-QQ},
where it is used to produce scalar decompositions having bitlength at
most \(\lceil{\log_2p}\rceil\).
While this basis generates a sublattice of determinant
\(\#\EC(\FF_{p^2})\),
if \(N \not= \#\EC(\FF_{p^2})\)
then the basis may be easily shrunk to fit \(\G\) following the procedure
outlined in \S\ref{sec:shrink}.

\section{
    Four-dimensional decompositions for GLV+GLS
}
\label{sec:LS}

Recently, Longa and Sica~\cite{Longa--Sica} 
followed by Guillevic and Ionica~\cite{Guillevic--Ionica}
have proposed using a pair of efficiently computable endomorphisms 
\(\phi\) and \(\psi\) to construct four-dimensional scalar
decompositions on elliptic curves,
corresponding to the situation of \S\ref{sec:general}
with \((\phi_1,\phi_2,\phi_3,\phi_4) = (1,\phi,\psi,\phi\psi)\).
(The Longa--Sica technique supersedes
the earlier 3-dimensional construction of Zhou, Hu, Xu, and Song~\cite{Zhou--Hu--Xu--Song}
with \((\phi_1,\phi_2,\phi_3) = (1,\phi,\psi)\),
which we will not cover here.)
Longa and Sica propose an elaborate iterated Cornacchia algorithm to
derive their lattice basis.
In this section, we show that no lattice algorithms are necessary: 
we can generate four short lattice vectors using
Lemma~\ref{lemma:relations}.


Recall the Longa--Sica construction:
Let \(\EC_0/\FF_p\) be a GLV curve, with an efficiently computable
endomorphism \(\phi_0\), and Frobenius \(\pi_0\).
Applying the GLS construction (exactly as in \S\ref{sec:GLS})
to \(\EC_0\),
we obtain a twisted elliptic curve \(\EC'\) over~\(\FF_{p^2}\)
with an efficiently computable endomorphism \(\psi\)
corresponding to the twist of \(\pi_0\):
if \(\tau: \EC_0\otimes\FF_{p^2}\to\EC'\) is the twisting isomorphism,
then \(\psi = \tau\pi_0\tau^{-1}\).
The endomorphisms \(\psi\) and \(\pi'\)
satisfy
\[
    \cpol{\psi}(T) = \cpol{\pi_0}(T) = T^2 - \Tr{\pi_0}T + p 
    \quad
    \text{and}
    \quad
    \cpol{\pi'}(T) = T^2 - (2p-\Tr{\pi_0}^2)T + p^2 
    ,
\]
respectively.
Clearly, \(\ZZ[\psi]\) contains \(\ZZ[\pi']\):
the discriminants of the orders \(\ZZ[\pi']\) and \(\ZZ[\psi]\) are
\[
    \Disc{\psi} = \Tr{\pi_0}^2 - 4p
    \qquad
    \text{and}
    \qquad
    \Disc{\pi'} 
    = \Tr{\pi_0}^2(\Tr{\pi_0}^2 - 4p)
    = \Tr{\pi_0}^2\Disc{\psi}
    \ ,
\]
so the relative conductor of \(\ZZ[\pi']\) in \(\ZZ[\psi]\) 
is \(|\Tr{\pi_0}| < 2\sqrt{p}\).
As in vanilla GLS, we can write
\[
    \label{eq:pi-psi-GLV+GLS}
    \pi' = -\Tr{\pi_0}\psi + p 
    \ .
\]

The GLV endomorphism \(\phi_0\) on \(\EC_0\) induces a second
efficient endomorphism 
\(\phi := \tau\phi_0\tau^{-1}\)
on~\(\EC'\).
We have \(\cpol{\phi} = \cpol{\phi_0}\),
so \(\ZZ[\phi] \cong \ZZ[\phi_0]\).
Since \(\phi_0\) is a GLV endomorphism,
\(\ZZ[\phi]\) is either the maximal order of 
the endomorphism algebra of \(\EC'\),
or very close to it---so it makes sense 
to assume that \(\ZZ[\phi]\) contains \(\ZZ[\psi]\)
(cf.~Remark~\ref{rem:LS-hyp} below),
so that we can write \(\psi\) as 
\[
    \label{eq:psi-phi-GLV+GLS}
    \psi = c\phi + b 
    \ ,
\]
where
\begin{equation}
    \label{eq:b-c-GLV+GLS}
    b = \frac{1}{2}(\Tr{\pi_0} - c\Tr{\phi})
    \qquad
    \text{and}
    \qquad 
    c^2 
    = 
    \frac{\Delta_\psi}{\Delta_\phi}
    =
    \frac{\Tr{\pi_0}^2-4p}{\Tr{\phi}^2-4\N{\phi}}
    \ .
\end{equation}

Observe that \(b\) and \(c\) are both in \(O(\sqrt{p})\).
As with conventional GLV curves, \(c\) and \(b\) 
are both already known as
byproducts of the curve construction (via the CM method) 
or its order computation:
if \(\pi_0 = c_0\phi_0 + b_0\),
then
\[
    c = \frac{c_0}{t_0}
    \quad
    \text{and}
    \quad
    b = \frac{1}{t_0}\left(b_0 - p\right)
    \ .
\]

\begin{theorem}
    \label{th:GLV+GLS}
    With \(\phi\) and \(\psi\) defined as above,
    suppose we are in the situation of \S\ref{sec:general}
    with \((\phi_1,\phi_2,\phi_3,\phi_4) = (1,\phi,\psi,\phi\psi)\).
    The vectors
    \begin{align*}
        \vv{b}_1 &= (1,0,b,c)
        \ ,
        &
        \vv{b}_2 &= (0,1,-c\N{\phi},c\Tr{\phi} + b)
        \ ,
        \\
        \vv{b}_3 &= (-b, -c, 1, 0)
        \ , 
        &
        \vv{b}_4 &= (c\N{\phi}, -c\Tr{\phi} - b, 0, 1)
    \end{align*}
    generate a sublattice 
    of determinant \(\#\EC(\FF_{q})\)
    in \(\Lattice\).
    If \(\G = \EC(\FF_{q})\),
    then \(\Lattice = \subgrp{\vv{b}_1,\vv{b}_2,\vv{b}_3,\vv{b}_4}\).
\end{theorem}
\begin{proof}
    Let \(\lambda_\phi\) and \(\lambda_\psi\) be the eigenvalues of
    \(\phi\) and \(\psi\) on \(\G\), respectively.
    Applying Lemma~\ref{lemma:relations}
    to the inclusion \(\ZZ[\psi] \subset \ZZ[\phi]\),
    we obtain relations
    \begin{align*}
        \lambda_\psi - c\lambda_\phi - b 
        & \equiv 0 \pmod{N}
        \qquad
        \text{and}
        \\
        \lambda_\psi\lambda_\phi - \Tr{\phi}\lambda_\psi - b\lambda_\phi +
        c\N{\phi} + b\Tr{\phi}
        & \equiv 0 \pmod{N}
        \ ,
    \end{align*}
    corresponding to the vectors \(\vv{b}_3\)
    and 
    \((c\N{\phi}+b\Tr{\phi},-c,-\Tr{\phi},1) = \vv{b}_4 - \Tr{\phi}\vv{b}_3\).
    Multiplying the relations above through by \(-\lambda_\psi\)
    and using \(\lambda_\psi^2 = \lambda_{\psi^2} = -1 \pmod{N}\),
    we obtain new relations
    \begin{align*}
        1 + c\lambda_\phi\lambda_\psi + b\lambda_\psi 
        & \equiv 0 \pmod{N}
        \qquad \text{and}
        \\
        \lambda_\phi - \Tr{\phi} + b\lambda_\phi\lambda_\psi
            - (c\N{\phi} + b\Tr{\phi})\lambda_\psi
        & \equiv 0 \pmod{N}
        \ ,
    \end{align*}
    corresponding to the vectors
    \(\left(1, 0, b, c\right) = \vv{b}_1\)
    and
    \(
        \left(
            -\Tr{\phi}, 1, -c\N{\phi} - b\Tr{\phi}, -b
        \right)
        = \vv{b}_4 - \Tr{\phi}\vv{b}_3
    \),
    respectively.
\end{proof}


The vectors produced by Theorem~\ref{th:GLV+GLS} are short:
\(\phi\) is a GLV endomorphism,
so both \(n_\phi\) and \(t_\phi\) are in \(O(1)\).
Hence,
in view of Eq.~\eqref{eq:b-c-GLV+GLS},
\(\|\vv{b}_i\|_\infty\) is in \(O(\sqrt{p})\)
for \(1 \le i \le 4\).

\begin{remark}
    \label{rem:LS-hyp}
    The assumption that \(\ZZ[\psi]\) is contained in \(\ZZ[\phi]\)
    is incompatible
    with the hypothesis of~\cite{Longa--Sica}
    (which supposes that \(\QQ(\phi)\cap\QQ(\psi) = \QQ\));
    but even without this assumption,
    \(\QQ(\phi) = \QQ(\psi) = \QQ(\pi')\)
    when \(\EC\) is ordinary,
    so the hypothesis of~\cite{Longa--Sica} is \emph{never} satisfied for
    ordinary curves.
\end{remark}

\section{Decompositions for the Guillevic--Ionica construction}
\label{sec:GI}
The Guillevic--Ionica construction~\cite{Guillevic--Ionica}
uses a modified CM method 
to search for elliptic curves \(\EC/\FF_{p^2}\) such that 
\(\EC\) has endomorphisms \(\phi\) and \(\psi\)
such that \(\phi\) is separable of very small degree \(d_1\),
and \(\psi\) is the composition of an inseparable \(p\)-isogeny
and a separable isogeny of very small degree \(d_2\).
(In a sense,
these curves are to Longa--Sica curves what 
reductions of \(\QQ\)-curves are to GLS curves.)
Once such \(\EC\) and \(p\) have been found (given \(d_1\) and \(d_2\)),
the \(\phi\) and \(\psi\) are 
easily recovered using V\'elu's formul\ae{}~\cite{Velu}.
If this construction is used, 
then (as with the standard CM method)
the expression of \(\pi\) as an element of \(\ZZ[\phi]\) is known:
\[
    \pi = c\phi + b 
    \ .
\]


\begin{theorem}
    \label{th:G-I}
    With \(\phi\) and \(\psi\) defined as above:
    Suppose we are in the situation of \S\ref{sec:general}
    with \((\phi_1,\phi_2,\phi_3,\phi_4) = (1,\phi,\psi,\phi\psi)\),
    and \(\psi^2 = [\pm d]\pi\).
    The vectors
    \begin{align*}
        \vv{b}_1 &= (\pm d,0,-b,-c)
        \ ,
        &
        \vv{b}_2 &= (0,\pm d,c\N{\phi},-c\Tr{\phi} - b)
        \ ,
        \\
        \vv{b}_3 &= (-b, -c, 1, 0)
        \ , 
        &
        \vv{b}_4 &= (c\N{\phi}, -c\Tr{\phi} - b, 0, 1)
    \end{align*}
    generate a sublattice of determinant \(\#\EC(\FF_{p^2})\) in
    \(\Lattice\).  If \(\G = \EC(\FF_{p^2})\), 
    then \(\Lattice = \subgrp{\vv{b}_1, \vv{b}_2, \vv{b}_3, \vv{b}_4}\).
\end{theorem}
\begin{proof}
    The proof is the same as for Theorem~\ref{th:GLV+GLS},
    but with~\(\lambda_\psi^2 = \pm d\).
\end{proof}

%
%
%
%
%

\section{
    Two-dimensional decompositions in genus 2
}
\label{sec:Kohel--Smith-T}

Suppose \(\A/\FF_{q}\) is an ordinary principally polarized 
abelian surface (in our applications, \(\A\) is either
the Jacobian of a genus 2 curve,
or the Weil restriction of an elliptic curve).
The Frobenius endomorphism \(\pi\) of \(\A\) generates 
a quartic CM field \(\QQ(\pi)\),
and the Rosati involution of \(\End(\A)\)
(exchanging an endomorphism with its Rosati dual)
acts as complex conjugation on \(\QQ(\pi)\).
Hence, the quadratic real subfield of \(\QQ(\pi)\)
is \(\QQ(\pi + \pidual)\),
and \(\ZZ[\pi + \pidual]\)
is a real quadratic order.
We may identify \(\pidual\) with \(q/\pi\);
the eigenvalue of \(\pi + \dualof{\pi}\) on subgroups of \(\A(\FF_{q})\)
is \(1 + q\) 
(because \(\pi\) has eigenvalue \(1\)).

The characteristic polynomial of Frobenius is
\[
    \cpol{\pi}(T) = 
    T^4 - \Tr{(\pi + \pidual)}T^3 + (2q + \N{(\pi+\pidual)})T^2 
    - \Tr{(\pi + \pidual)}T + q^2 
    \ ,
\]
so
\begin{equation}
    \label{eq:g2-RM-card}
    \#\A(\FF_q)
    = 
    \cpol{\pi}(1)
    =
    \cpol{\pi+\pidual}(q+1)
    =
    (q + 1)^2 - \Tr{(\pi + \pidual)}(q + 1) + \N{(\pi+\pidual)}
    \ .
\end{equation}
The Weil bounds yield
\begin{align*}
    |\Tr{(\pi+\pidual)}| \le 4\sqrt{q}
    \quad &\text{and}\quad
    |\N{(\pi+\pidual)}| \le 4q
    \ ,
\intertext{while R\"uck~\cite{Ruck} shows that}
    \Tr{(\pi+\pidual)}^2 - 4\N{(\pi+\pidual)} > 0
    \quad &\text{and}\quad 
    \N{(\pi+\pidual)} + 4q > 2|\Tr{(\pi+\pidual)}|\sqrt{q}
    \ .
\end{align*}

\begin{theorem}
    \label{th:RM-2d}
    Suppose \(\phi\) is a non-integer real multiplication endomorphism
    of an ordinary abelian surface \(\A\) (ie, \(\dualof{\phi} = \phi\))
    such that \(\ZZ[\pi + \pidual]\subseteq\ZZ[\phi]\),
    and assume that we are in the situation of \S\ref{sec:general}
    with \((\phi_1,\phi_2) = (1,\phi)\).
    If \(\pi + \pidual = c\phi + b\),
    then the vectors
    \[
        \vv{b}_1 
        =
        \left(
            q+1 - b, -c
        \right)
        \qquad
        \text{and}
        \qquad
        \vv{b}_2 
        =
        \left(
            c\N{\phi} - (q+1 - b)\Tr{\phi}, q+1 - b
        \right)
    \]
    generate a sublattice 
    of determinant \(\#\A(\FF_{q})\)
    in \(\Lattice\).
    If \(\G = \A(\FF_{q})\),
    then \(\Lattice = \subgrp{\vv{b}_1,\vv{b}_2}\).
\end{theorem}
\begin{proof}
    The proof is almost exactly the same as that of
    Theorem~\ref{th:EC-2d}.
    As we noted above,
    \((\pi + \pidual)\) has eigenvalue \(q+1\) on \(\G\).
    Applying Lemma~\ref{lemma:relations}
    to \(\ZZ[\pi + \pidual] \subset \ZZ[\phi]\),
    we obtain relations
    \begin{align*}
        (q+1 - b)\cdot 1  - c\cdot \lambda_{\phi} 
        & \equiv 0 \pmod{N} 
        \qquad \text{and} \\
        (c\N{\phi} - (q+1-b)\Tr{\phi})\cdot 1 + (q+1-b)\cdot\lambda_\phi 
        & \equiv 0 \pmod{N}
        \ .
    \end{align*}
    The first implies that \(\vv{b}_1\) is in \(\Lattice\),
    the second that \(\vv{b}_2\) is in \(\Lattice\).
    The determinant of \(\subgrp{\vv{b}_1,\vv{b}_2}\)
    is 
    \[
        (q+1-b)^2 + c^2\N{\phi} - (q+1-b)c\Tr{\phi}
        = 
        \#\A(\FF_{q})
        \ ,
    \]
    using Eq.~\eqref{eq:g2-RM-card},
    \(\Tr{(\pi + \pidual)} = c\Tr{\phi} + 2b\),
    and \(\Tr{\phi}^2 - 4\N{\phi} = c^2(\Tr{(\pi + \pidual)}^2 -
    4\N{(\pi + \pidual)})\).
\end{proof}



We apply Theorem~\ref{th:RM-2d} to the explicit real multiplication families
treated in~\cite{Takashima,Kohel--Smith,Gaudry--Kohel--Smith}.
In each case,
the assumption \(\ZZ[\pi+\pidual] \subseteq \ZZ[\phi]\) is fulfilled
because \(\ZZ[\phi]\) is the maximal order of the real subfield of 
the endomorphism algebra.


\begin{example}[Explicit RM by \(\QQ(\sqrt{5})\)]
    Brumer~\cite{Brumer},
    Hashimoto~\cite{Hashimoto},
    Mestre~\cite{Mestre-1},
    and
    Tautz, Top, and Verberkmoes~\cite{Tautz--Top--Verberkmoes}
    have given explicit constructions
    of families of genus 2 curves 
    whose Jacobians have explicit real multiplication 
    by \(\ZZ[(1 + \sqrt{5})/2]\),
    which is the maximal order of \(\QQ(\sqrt{5})\)
    (see Wilson's thesis~\cite{Wilson}
    for a full characterization of all such curves).
    In each case, the curve is equipped with a correspondence
    inducing an explicit endomorphism \(\phi\) on the Jacobian
    satisfying
    \(\cpol{\phi}(T) = T^2 + T - 1\);
    these endomorphisms have been exploited for fast scalar
    multiplication in~\cite{Kohel--Smith} and~\cite{Takashima}.

    Let \(\C/\FF_{q}\) be the reduction mod \(p\) of a curve taken from
    one of these families;
    then \(\Jac{\C}\) inherits the explicit endomorphism \(\phi\),
    and \(\ZZ[\phi] \cong \ZZ[(1 + \sqrt{5})/2]\).
    Since \(\ZZ[(1 + \sqrt{5})/2]\)
    is the maximal order of \(\QQ(\sqrt{5})\),
    we must have \(\ZZ[\pi + \pidual] \subseteq \ZZ[\phi]\);
    so
    \[
        \pi + \pidual = c\phi + b
    \]
    where \(b = \frac{1}{2}(\Tr{(\pi+\pidual)} + c)\)
    and
    \(5c^2 = \Disc{\pi+\pidual}\).

    Putting ourselves in the situation of~\S\ref{sec:general}
    with \((\phi_1,\phi_2) = (1,\phi)\),
    Theorem~\ref{th:RM-2d}
    yields vectors
    \[
        \vv{b}_1 
        =
        \left(
            q+1 - b, -c
        \right)
        \qquad
        \text{and}
        \qquad
        \vv{b}_2 
        =
        \left(
            -c - (q+1 - b), q+1 - b
        \right)
    \]
    in \(\Lattice\).
    Note that \(|c| < 4\sqrt{q/5}\),
    so \(|b| < (2 + 2/\sqrt{5})\sqrt{q}\),
    and
    \[
        \sigma(\vv{b}_1) = \sigma(\vv{b}_2) = \log_2(q+1) 
        .
    \]
\end{example}

\begin{example}
    Mestre~\cite{Mestre-2}
    has constructed a two-parameter family of genus 2 curves
    whose Jacobians have explicit real multiplication by
    \(\ZZ[\sqrt{2}]\)
    (an alternative presentation of these curves for cryptographic
    applications is developed
    in~\cite{Gaudry--Kohel--Smith};
    see Bending's thesis~\cite{Bending} for a full characterization
    of curves with RM by \(\ZZ[\sqrt{2}]\)).
    The efficient endomorphism \(\phi\)
    satisfies \(\cpol{\phi}(T) = T^2 - 2\) in this case,
    so \(\Disc{\phi} = 8\).

    Let \(\C/\FF_{q}\) be the reduction mod \(p\) of a curve taken from
    one of these families;
    \(\Jac{\C}\) inherits the explicit endomorphism \(\phi\),
    and \(\ZZ[\phi] \cong \ZZ[\sqrt{2}]\).
    Since \(\ZZ[\sqrt{2}]\) is the maximal order of \(\QQ(\sqrt{2})\),
    we must have \(\ZZ[\pi + \pidual] \subseteq \ZZ[\phi]\);
    so
    \[
        \pi + \pidual = c\phi + b
    \]
    where \(b = \frac{1}{2}(\Tr{(\pi+\pidual)} + c)\)
    and
    \(2c^2 = \Disc{\pi+\pidual}\).

    Putting ourselves in the situation of~\S\ref{sec:general}
    with \((\phi_1,\phi_2) = (1,\phi)\),
    Theorem~\ref{th:RM-2d}
    yields vectors
    \[
        \vv{b}_1 
        =
        \left(
            q+1 - b, -c
        \right)
        \qquad
        \text{and}
        \qquad
        \vv{b}_2 
        =
        \left(
            -2c, q+1 - b
        \right)
    \]
    in \(\Lattice\).
    For this family,
    \(|c| < 2\sqrt{q/2}\) and \(|b| < (2 + 1/2\sqrt{2})\sqrt{q}\);
    each is much smaller than \(q+1\),
    so as before we have
    \[
        \sigma(\vv{b}_1) = \sigma(\vv{b}_2) = \log_2(q+1) 
        .
    \]
\end{example}



\end{document}